\documentclass[a4paper,12pt]{article}

\input{mymacros.sty}

\newcommand{\nof}{r}
\newcommand{\SR}{\operatorname{SR}}
\newcommand{\bone}{\bsone}

\title{Tropical coamoeba and torus-equivariant
homological mirror symmetry for the projective space}
\author{Masahiro Futaki and Kazushi Ueda}
\date{}
\begin{document}

\maketitle

\begin{abstract}
We introduce the notion of a tropical coamoeba
which gives a combinatorial description
of the Fukaya category of the mirror of a toric Fano stack.
We show that the polyhedral decomposition of a real $n$-torus
into $n+1$ permutohedra gives a tropical coamoeba
for the mirror of the projective space $\bP^n$,
and prove a torus-equivariant version of
homological mirror symmetry for the projective space.
As a corollary, we obtain homological mirror symmetry
for toric orbifolds of the projective space.
\end{abstract}

\section{Introduction}
 \label{sc:introduction}

Let $n$ be a natural number and
$\Delta$ be a convex lattice polytope in $\bR^n$,
i.e., 
the convex hull of a finite subset of $\bZ^n$.
We assume that the origin is in the interior of $\Delta$.
One side of homological mirror symmetry for toric Fano stacks,
conjectured by Kontsevich \cite{Kontsevich_ENS98},
states that there is an equivalence
\begin{equation} \label{eq:hms}
 D^b \coh X \cong D^b \Fuk W
\end{equation}
of two triangulated categories of geometric origins
associated with $\Delta$.
For the other side of homological mirror symmetry for toric manifolds,
we refer the readers to the survey paper
\cite{MR3076063} and references therein.

The category on the left hand side is
the derived category of coherent sheaves
on the toric Fano stack $X$,
defined as follows:
Let $\{ v_i \}_{i=1}^\nof$
be the set of vertices of $\Delta$ and 
take a simplicial stacky fan $\bsSigma$
such that the set of generators of one-dimensional cones
is given by $\{ v_i \}_{i=1}^\nof$.
The associated toric stack is the quotient stack
\begin{equation*}
 X = [(\bC^\nof \setminus \SR(\Sigma)) / K],
\end{equation*}
where the Stanley-Reisner locus $\SR(\Sigma)$
consists of points $(z_1, \dots, z_r)$ such that
there is no cone in $\Sigma$ which contains
all $v_i$ for which $z_i = 0$, and
$$
 K = \Ker (\phi \otimes \bCx)
$$
is the kernel of the tensor product with $\bCx$
of the map
$
 \phi :  \bZ^\nof \to  \bZ^n
$
sending the $i$-th coordinate vector to $v_i$
for $i = 1, \dots, \nof$.
Although 
$X$ depends not only on $\Delta$
but also on $\Sigma$,
the derived category $D^b \coh X$ is independent of this choice
\cite[Theorem 4.2]{Kawamata_LCBMDC}
and depends only on $\Delta$.

On the right hand side,
one takes a sufficiently general Laurent polynomial
$$
 W = \sum_{\omega \in \Delta \cap \bZ^n} a_\omega x^\omega
$$
whose Newton polytope coincides with $\Delta$
as in \cite{Givental_HGMS}.
This defines an exact Lefschetz fibration
$$
 W : (\bCx)^n \to \bC
$$
with respect to the standard cylindrical K\"{a}hler structure
on $(\bCx)^n$,
and $\Fuk W$ is the directed Fukaya category
defined by Seidel \cite{Seidel_VC2, Seidel_PL} 
whose set of objects is a distinguished basis of vanishing cycles and
whose spaces of morphisms are Lagrangian intersection Floer complexes.

The equivalence \eqref{eq:hms} is proved
for $\bP^2$ and $\bP^1 \times \bP^1$ by Seidel \cite{Seidel_VC2},
weighted projective planes and Hirzebruch surfaces
by Auroux, Katzarkov and Orlov
\cite{Auroux-Katzarkov-Orlov_WPP},
toric del Pezzo surfaces by Ueda \cite{Ueda_HMSTdPS}, and
toric orbifolds of toric del Pezzo surfaces
by Ueda and Yamazaki
\cite{Ueda-Yamazaki_toricdP}.
See also Auroux, Katzarkov and Orlov
\cite{Auroux-Katzarkov-Orlov_dP}
for homological mirror symmetry for
not necessarily toric del Pezzo surfaces,
Abouzaid \cite{Abouzaid_HCRMSTV, Abouzaid_MHTGHMSTV}
for an application of tropical geometry
to homological mirror symmetry,
Kerr \cite{Kerr_WBMSTS}
for the behavior of homological mirror symmetry
under weighted blowup of toric surfaces.
Slightly different versions
of homological mirror symmetry for toric stacks
are proved by
Fang, Liu, Treumann and Zaslow
\cite{Fang_HMSTP, Fang-Liu-Treumann-Zaslow_CCC,
MR3168399}
and Futaki and Ueda
\cite{Futaki-Ueda_A-infinity}.

In this paper,
we pass to the universal cover
$$
 \exp : \bC^n \to (\bCx)^n
$$
of the torus and
replace the Lefschetz fibration $W$
with its pull-back
$$
 \Wtilde = W \circ \exp : \bC^n \to \bC.
$$
The fact that $\Wtilde$ has countably many critical points
does not cause any problem, and
one can formulate a torus-equivariant version of
homological mirror symmetry for toric Fano stacks:

\begin{conjecture} \label{conj:equiv_hms}
For a convex lattice polytope $\Delta$
containing the origin in its interior,
there is an equivalence
$$
 D^b \coh^\bT X \cong D^b \Fuk \Wtilde
$$
of triangulated categories.
\end{conjecture}

Here $\bT$ is the $n$-dimensional torus acting on $X$ and
$D^b \coh^\bT X$ is the derived category of
$\bT$-equivariant coherent sheaves on $X$.
Our first main result is the proof of
Conjecture \ref{conj:equiv_hms}
for the projective space:

\begin{theorem} \label{th:equiv_hms}
Conjecture \ref{conj:equiv_hms} holds
when $X$ is the projective space.
\end{theorem}

The case $n = 3$ in Theorem \ref{th:equiv_hms} is a corollary
of a result of Seidel \cite[Proposition 11.7]{Seidel_K3}
which describes the undirected Fukaya category
of the fiber of $W$.
Theorem \ref{th:equiv_hms} is an important step
in the proof of homological mirror symmetry
for the quintic 3-fold in \cite{Nohara-Ueda_HMSQ}.
Torus-equivariant homological mirror symmetry for $X$ implies
the ordinary homological mirror symmetry,
not only for $X$
but also for the quotient stack $[X / A]$
for any finite subgroup $A$ of the torus $\bT$
acting on $X$.

\begin{corollary} \label{cor:hms}
For a convex lattice polytope $\Delta$
which can be obtained from the polytope for $\bP^n$
by an integral linear transformation,
one has an equivalence
$$
 D^b \coh X \cong D^b \Fuk W
$$
of triangulated categories.
\end{corollary}

We introduce the notion of a {\em tropical coamoeba} of $W$,
which consists of a decomposition
$$
 T = \bigcup_{i=1}^m P_i
$$
of a real $n$-torus $T = \bR^n / \bZ^n$
into the union of an ordered set of polytopes,
together with a map
$$
 \deg : F_1 \to \bZ
$$
from the set $F_1$ of facets of $P_i$ to $\bZ$ called the {\em degree},
and a map
$$
 \sgn : F_2 \to \{ 1, - 1 \}
$$
from the set $F_2$ of codimension two faces
of $P_i$ called the {\em sign},
satisfying conditions
in Definition \ref{def:tropical_coamoeba}.
One can associate a directed $A_\infty$-category
with a tropical coamoeba,
and the conditions in Definition \ref{def:tropical_coamoeba}
ensure that
this $A_\infty$-category is equivalent to $\Fuk \Wtilde$.
This enables us to divide Conjecture \ref{conj:equiv_hms}
into two steps:

\begin{conjecture} \label{conj:tropical_coamoeba}
Let $\Delta$ be a convex lattice polytope in $\bR^n$
containing the origin in its interior.
Then the following hold:
\begin{itemize}
 \item
There is a Laurent polynomial $W : (\bCx)^n \to \bC$ such that
\begin{itemize}
 \item
the Newton polytope of $W$ coincides with $\Delta$, and
 \item
there exists a tropical coamoeba $G$ of $W$.
\end{itemize}
This implies that the $A_\infty$-category
$\scA_\Gtilde$ associated with $G$ is
quasi-equivalent to $\Fuk \Wtilde$;
$$
 \scA_\Gtilde \cong \Fuk \Wtilde.
$$
 \item
The derived category of the $A_\infty$-category
$\scA_\Gtilde$ is equivalent
to the derived category of $\bT$-equivariant coherent sheaves
on the toric Fano stack $X$ associated with $\Delta$;
$$
 D^b \scA_\Gtilde \cong D^b \coh^\bT X.
$$
\end{itemize}
\end{conjecture}


Our second main result is
the proof of Conjecture \ref{conj:tropical_coamoeba}
for the projective space:

\begin{theorem} \label{th:coamoeba}
Conjecture \ref{conj:tropical_coamoeba} holds
when $X$ is the projective space.
The tropical coamoeba in this case
comes from a decomposition of a real $n$-torus
into the union of $n + 1$ permutohedra of order $n + 1$.
\end{theorem}

A tropical coamoeba is a generalization of a dimer model
to higher dimensions.
The importance of dimer models
in mirror symmetry
is pointed out by Feng, He, Kennaway and Vafa
\cite{Feng-He-Kennaway-Vafa} and
elaborated in \cite{Ueda-Yamazaki_NBTMQ,
Ueda-Yamazaki_BTP,
Ueda-Yamazaki_toricdP}.
The works of Bondal and Ruan
\cite{Bondal_DCTV}
and Fang, Liu, Treumann and Zaslow
\cite{Fang_HMSTP, Fang-Liu-Treumann-Zaslow_CCC,
MR3168399}
use constructible sheaves on a real torus
and its universal cover to study equivariant homological mirror symmetry
for toric stacks, and
it is an interesting problem to explore relationship
between their approach and ours.

For any convex lattice polytope $\Delta$ in $\bR^n$
and a vertex $v$ of $\Delta$,
one can obtain another polytope $\Delta'$
by removing $v$ from $\Delta \cap \bZ^n$ and
taking the convex hull of the rest.
On the complex side,
this operation gives a birational map
$X_\Delta \dashrightarrow X_{\Delta'}$
between the corresponding toric stacks,
which in turn gives a full and faithful functor
$$
 \Phi : D^b \coh X_{\Delta'} \hookrightarrow D^b \coh X_\Delta
$$
by a result of Kawamata \cite[Theorem 4.2]{Kawamata_LCBMDC}.
On the symplectic side,
one can choose a one-parameter family $W_t$ of Laurent polynomials
such that the Newton polytope of $W_0$ is $\Delta'$ and
that of $W_t$ for $t \ne 0$ is $\Delta$,
so that a result of Kerr \cite[Theorem 6]{Kerr_WBMSTS}
gives a full and faithful functor
$$
 \Psi : D^b \Fuk W_{\Delta'} \hookrightarrow D^b \Fuk W_\Delta.
$$
It is clear that any lattice polytope can be embedded
into a sufficiently large simplex,
so that any lattice polytope can be obtained
from a sufficiently large simplex
by successively performing this operation.
As a corollary,
one obtains the following:

\begin{corollary} \label{cor:embedding}
For any lattice polytope $\Delta$,
there exist a Laurent polynomial $W'$ and a toric stack $X'$
such that one has full and faithful functors
$
 F : D^b \coh X \hookrightarrow D^b \Fuk W'
$
and
$
 G : D^b \Fuk W \hookrightarrow D^b \coh X',
$
where $X$ is a toric stack associated with $\Delta$ and
$W$ is a general Laurent polynomial
whose Newton polytope is $\Delta$.
\end{corollary}

The lattice polytope associated with $X'$
and the Newton polytope of $W'$
in Corollary \ref{cor:embedding}
are sufficiently large simplexes containing $\Delta$.
As the equivalence in Corollary \ref{cor:hms}
is given explicitly,
one can in principle reduce homological mirror symmetry
for a general toric stack
to the problems of
\begin{itemize}
 \item
the behavior of the derived category of toric stacks
under birational trasformations, and
 \item
the behavior of critical values of Laurent polynomials
under deformations,
\end{itemize}
without any further Floer-theoretic computations
on vanishing cycles.
This is a special case
of the relation between homological mirror symmetry
and the minimal model program
discussed in \cite{Ballard-Favero-Katzarkov_VGITDC,
Diemer-Katzarkov-Kerr_SGR}.

The organization of this paper is as follows:
We collect basic definitions on Fukaya categories
in Section \ref{sc:fuk}.
Symplectic Picard-Lefschetz theory
and homological mirror symmetry for $\bP^2$
by Seidel are recalled in Section \ref{sc:PL} and Section \ref{sc:p2}
respectively,
which are used in Section \ref{sc:p3} to prove
homological mirror symmetry for $\bP^3$.
The Fukaya category of the mirror of $\bP^n$ for general $n$ is
computed in Section \ref{sc:induction}
by an induction on $n$.
In Section \ref{sc:coamoeba},
we define a tropical coamoeba as a combinatorial object
which encode the information of the Fukaya category,
and show that it allows one to summarize
the result in Section \ref{sc:induction}
in a nice way.

{\em Acknowledgment}:
K.~U. thanks
Alexander Esterov,
Akira Ishii and
Dominic Joyce
for useful discussions and remarks.
We also thank the anonymous referee
for suggesting several improvements.
M.~F. is supported by Grant-in-Aid for Young Scientists (No.19.8083).
K.~U. is supported by Grant-in-Aid for Young Scientists (No.18840029).
This work has been done while K.~U. is visiting the University of Oxford,
and he thanks the Mathematical Institute for hospitality
and Engineering and Physical Sciences Research Council
for financial support.

\section{Fukaya categories}
 \label{sc:fuk}

For a $\bZ$-graded vector space
$N = \oplus_{j \in \bZ} N^j$ and an integer $i$,
the $i$-th shift of $N$ to the left
will be denoted by $N[i]$;
$(N[i])^j  = N^{i+j}$.

\begin{definition} \label{def:a_infinity}
An $A_\infty$-category $\scA$ consists of
\begin{itemize}
 \item the set $\Ob(\scA)$ of objects,
 \item for $c_1,\; c_2 \in \Ob(\scA)$,
       a $\bZ$-graded vector space $\hom_\scA(c_1, c_2)$
       called the space of morphisms, and
 \item operations
$$\
 \m_l : \hom_\scA (c_{l-1},c_l) \otimes \dots
          \otimes \hom_\scA (c_0,c_1)
 \longrightarrow \hom_\scA (c_0,c_l)
$$
of degree $2 - l$ for $l=1,2,\ldots$
and $c_i \in \Ob(\scA)$,
$i=0,\dots, l$,
satisfying
the {\em $A_\infty$-relations}
\begin{eqnarray}
 \sum_{i=0}^{l-1} \sum_{j=i+1}^l
  (-1)^{\deg a_1 + \cdots + \deg a_i - i}
  \m_{l+i-j+1}(a_l \otimes \cdots \otimes a_{j+1}
   \otimes 
    \m_{j-i}(a_j \otimes \cdots \otimes a_{i+1})
     \nonumber \\
   \otimes
    a_i \otimes \cdots \otimes a_1 ) = 0,
  \label{eq:A_infty}
\end{eqnarray}
for any positive integer $l$,
any sequence $c_0, \dots, c_l$ of objects of $\scA$,
and any sequence of morphisms
$a_m \in \hom_{\scA}(c_{m-1}, c_m)$
for $m = 1, \dots, l$.
\end{itemize}
\end{definition}


The $A_\infty$-relations \eqref{eq:A_infty}
for $l = 1$, $2$, and $3$ show that
$\m_1$ squares to zero and
$\m_2$ defines an associative operation
on the cohomology of $\m_1$.
The resulting non-unital category is called
the {\em cohomological category} of $\scA$.
An $A_\infty$-category
satisfying
$\m_k = 0$ for $k \ge 3$
corresponds to a {\em differential graded category}
(i.e. a category whose spaces of morphisms are complexes
such that the differential $d$ satisfies the Leibniz rule
with respect to the composition)
by
$$
 d(a) = (-1)^{\deg a} \frakm_1(a), \qquad
 a_2 \circ a_1 = (-1)^{\deg a_1} \frakm_2(a_2, a_1).
$$
The derived category of an $A_\infty$-category
is defined using {\em twisted complexes},
which are introduced
by Bondal and Kapranov
\cite{Bondal-Kapranov_ETC}
for differential graded categories
and generalized to $A_\infty$-categories
by Kontsevich \cite{Kontsevich_HAMS}.
Here we follow the exposition of Seidel
\cite{Seidel_PL} closely.
For an $A_\infty$-category $\scA$,
its {\em additive enlargement} $\Sigma \scA$
is the $A_\infty$-category
whose set of object consists of formal direct sums
$$
 X = \bigoplus_{i \in I} V^i \otimes X^i
$$
where $I$ is a finite set,
$\{ X^i \}_{i \in I}$ is a family of objects of $\scA$, and
$\{ V^i \}_{i \in I}$ is a family of graded vector spaces.
The space of morphisms is given by
$$
 \hom_{\Sigma \scA}
  \lb
   \bigoplus_{i \in I_0} V_0^i \otimes X_0^i,
   \bigoplus_{i \in I_1} V_1^i \otimes X_1^i,
  \rb
  = \bigoplus_{i, j}
     \hom_{\bC}(V_0^i, V_1^j) \otimes \hom_{\scA}(X_0^i, X_1^j)
$$
and the $A_\infty$-operations are
$$
 \frakm_d^{\Sigma \scA}(a_d, \dots, a_1)^{i_d, i_0}
  = \sum_{i_1, \dots, i_d} (-1)^\dagger
     \phi_d^{i_d, i_{d-1}} \circ \cdots \circ \phi_1^{i_1, i_0}
     \otimes \mu_d^{\scA}(x_d^{i_d, i_{d-1}}, \ldots, x_1^{i_1, i_0}),
$$
where
$
 \dagger
  = \sum_{p < q}
     \deg \phi_p^{i_p, i_{p-1}}
      \cdot (\deg x_q^{i_q, i_{q-1}} - 1)
$
and
$
 a_k = (a_k^{ji}) = (\phi_k^{ji} \otimes x_k^{ji}).
$

A twisted complex is a pair
$$
 \lb
  X = \bigoplus_{i \in I} V^i \otimes X^i,
  \delta_X = (\delta_X^{ji})
 \rb
$$
of an object $X$ of $\Sigma \scA$ and a morphism
$\delta \in \hom^1_{\Sigma \scA}(X, X)$,
satisfying the {\em Maurer-Cartan equation}
$$
 \sum_{i=1}^\infty
  \frakm_r^{\Sigma \scA}(\delta_X, \ldots, \delta_X)
 = 0.
$$
Twisted complexes constitute an $A_\infty$-category $\Tw \scA$,
whose $A_\infty$-operations are given by
\begin{eqnarray*}
 \frakm_d^{\Tw \scA}(a_d, \dots, a_1)
  = \sum_{i_0, \dots, i_d}
     \frakm_{d + i_0 + \dots + i_d}^{\Sigma \scA}
      (\overbrace{\delta_{X_d}, \dots, \delta_{X_d}}^{i_d}, a_d, \\
 \underbrace{\delta_{X_{d-1}}, \dots, \delta_{X_{d-1}}}_{i_{d-1}},
 a_{d-1}, \dots, a_1,
 \underbrace{\delta_{X_{0}}, \dots, \delta_{X_{0}}}_{i_{0}}),
\end{eqnarray*}
where the sum is over all $i_0, \dots, i_d \ge 0$.
The $A_\infty$-relations in $\Tw \scA$ comes from
that of $\scA$ and the Maurer-Cartan equation.
The cohomological category $D^b \scA$ of $\Tw \scA$ is triangulated,
and the mapping cone of a closed morphism
$c \in \hom^0_{\Tw \scA}(X_0, X_1)$ is defined by
$$
 \lb
  C = \bC[1] \otimes X_0 \oplus \bC \otimes X_1, \ 
  \delta_C = \begin{pmatrix}
   \bone_{1, 1} \otimes \delta_{X_0} & 0 \\
   - \bone_{1, 0} \otimes c & \bone_{0, 0} \otimes \delta_{X_1}
  \end{pmatrix}
 \rb
$$
where $\bone_{i, j} \in \hom_{\bC}(\bC[i], \bC[j])$
is the identity morphism of degree $i-j$.

The Fukaya category $\Fuk M$
of a symplectic manifold $(M, \omega)$
is an $A_\infty$-category
whose objects are Lagrangian submanifolds of $M$
(together with additional structures
such as gradings, spin structures and flat $U(1)$ bundles on them)
and whose spaces of morphisms are
Lagrangian intersection Floer complexes
\cite{Fukaya_MHACFH, Fukaya-Oh-Ohta-Ono, Seidel_PL}:
For two objects $L_1$ and $L_2$ intersecting transversely,
$\hom(L_1, L_2)$ is a graded vector space
spanned by intersection points
$L_1 \cap L_2$.
For a positive integer $k$,
a sequence
$(L_0, \ldots, L_k)$ of objects,
and morphisms $p_l \in L_{\ell-1} \cap L_{\ell}$
for $\ell = 1, \ldots, k$,
the $A_\infty$-operation $\m_k$ is given by
counting the virtual number of holomorphic disks
with Lagrangian boundary conditions;
$$
 \m_k(p_k,\ldots,p_1) = \sum_{p_0 \in L_0 \cap L_k}
     \# \scMbar_{k+1}(L_0, \ldots, L_k; p_0, \ldots, p_k) p_0.
$$
Here,
$
 \scMbar_{k+1}(L_0, \ldots, L_k; p_0, \ldots, p_k)
$
is the stable compactification of
the moduli space of holomorphic maps
$
 \phi : D^2 \to M
$
from the unit disk $D^2$
with $k + 1$ marked points
$(z_0, \ldots, z_k)$
on the boundary
respecting the cyclic order,
with the following boundary condition:
Let $\partial_l D^2 \in \partial D^2$
be the interval
between $z_l$ and $z_{l+1}$,
where we set $z_{k+1} = z_0$.
Then
$\phi(\partial_l D^2) \subset L_\ell$ and $\phi(z_l) = p_l$
for $\ell = 0, \dots, k$.

Let $M$ be a symplectic manifold and
$p : \Mtilde \to M$ be a regular covering
with the covering transformation group $G$,
so that there is an exact sequence
$$
 1 \to \pi_1(\Mtilde) \xto{p_*} \pi_1(M) \to G \to 1
$$
of groups.
Let
$
 i : L \hookrightarrow M
$
be a Lagrangian submanifold.
If the image of
$
 i_* : \pi_1(L) \to \pi_1(M)
$
is contained in the image of $p_*$,
then the set of connected components of $\Ltilde = p^{-1}(L)$
forms a torsor over $G$,
so that one has
$$
 \Ltilde = \coprod_{g \in G} \Ltilde_g
$$
for a choice of a connected component $\Ltilde_e \subset \Ltilde$.
Given a pair $(L, L')$ of such Lagrangian submanifolds,
one has an isomorphism
$$
 \hom_{\Fuk M}(L, L')
  \cong
   \bigoplus_{g \in G}
    \hom_{\Fuk \Mtilde}(\Ltilde_e, \Ltilde_g'),
$$
which is compatible with the $A_\infty$-operations.

\section{Symplectic Picard-Lefschetz theory}
 \label{sc:PL}

Let
$
 \pi : E \to \bC
$
be a holomorphic function
on a K\"{a}hler manifold $E$,
whose K\"{a}hler form is exact.
We assume that $E$ is complete
as a Riemannian manifold,
and $\| \nabla \pi \|$ is a proper function on $E$.
The mirrors of toric Fano stacks satisfy
these conditions
\cite[Example 6.1]{Seidel_suspension}.
The map $\pi$ is said to be
an {\em exact Lefschetz fibration}
if all the critical points of $\pi$ are non-degenerate.
This means that for any critical point $p \in E$,
one can choose a holomorphic local coordinate
$(x_1, \dots, x_n)$
of $E$ around $p$ such that
\begin{equation} \label{eq:quadratic}
 \pi(x_1, \dots, x_n) = x_1^2 + \cdots + x_n^2 + w,  
\end{equation}
where $w$ is the critical value of $\pi$.
For the moment,
we assume that all the critical values are distinct
and $0$ is a regular value of $\pi$.
We choose the origin as the base point
and write
$$
 E_0 = \pi^{-1}(0).
$$
A {\em vanishing path} is an embedded path
$\gamma : [0, 1 ] \to \bC$
such that
\begin{itemize}
 \item
$\gamma(0) = 0$,
 \item
$\gamma(1)$ is a critical value of $\pi$, and
 \item
$\gamma(t)$ is not a critical value of $\pi$
for $t \in (0, 1)$.
\end{itemize}
A {\em distinguished set of vanishing paths}
is an ordered set $(\gamma_i)_{i=1}^m$
of vanishing paths $\gamma_i: [0, 1] \to \bC$
such that
\begin{itemize}
 \item $\{ \gamma_i(1) \}_{i=1}^m$ is the set of critical values of $\pi$,
 \item images of $\gamma_i$ and $\gamma_j$ for $i \ne j$
intersect only at the origin,
 \item $\gamma_i'(0) \neq 0$ for $i = 1, \dots, m$, and
 \item $\arg \gamma_1'(0) > \cdots > \arg \gamma_m'(0)$
        for a suitable choice of a branch of the argument map.
\end{itemize}
Let $\gamma$ be a vanishing path
and $y$ be the critical point of $\pi$ above $\gamma(1)$.
The {\em vanishing cycle} along $\gamma$ is
the cycle of $E_0$
which collapses to the critical point $y$
by the symplectic parallel transport along $\gamma$;
$$
 V_\gamma = \lc x \in E_0 \left| \,
           \lim_{t \rightarrow 1} \gammatilde_x(t) = y \rc \right. .
$$
Here, the horizontal lift
$\gammatilde_x : [0,1) \to E$
of $\gamma : [0,1] \to \bC$
starting from $x \in E_0$ is defined
by the condition that
the tangent vector of the curve $\gammatilde$
is orthogonal to the tangent space of the fiber
with respect to the K\"{a}hler form.

The vanishing cycle $V_\gamma$ is a Lagrangian $(n-1)$-sphere $E_0$.
The trajectory
$$
 \Delta_\gamma = \bigcup_{x \in V_\gamma} \Image \gammatilde_x
$$
of the vanishing cycle is called the {\em Lefschetz thimble}.
It is a Lagrangian ball in $E$
whose boundary is the corresponding vanishing cycle;
$$
 \partial \Delta_\gamma = V_\gamma.
$$
For a distinguished set $(\gamma_i)_{i=1}^m$ of vanishing paths,
the ordered set
$$
 \bsV
  = (V_{\gamma_1}, \dots, V_{\gamma_n})
$$
is called the {\em distinguished basis of vanishing cycles}.

To define the Fukaya category of the Lefschetz fibration,
let
$$
 \beta :
  \Etilde = \{ (x, y) \in E \times \bC \mid \pi(x) = y^2 \}
    \to E
$$
be the double cover of $E$
branched along the fiber
$E_0 = \pi^{-1}(0)$ over the origin.
Then
the covering transformation
$
 \iota : (x, y) \mapsto (x, -y)
$
defines a $\bZ / 2 \bZ$-action on $\Etilde$,
which induces a $\bZ / 2 \bZ$-action on the Fukaya category
$\Fuk \Etilde$ of $\Etilde$.
Roughly speaking,
the Fukaya category $\frakF(\pi)$ of the Lefschetz fibration $\pi$ is
defined as the $\iota$-invariant part
of $\Fuk \Etilde$;
objects of $\frakF(\pi)$ are $\iota$-invariant
Lagrangian submanifolds of $\Etilde$,
and the space of morphisms in $\frakF(\pi)$
are $\iota$-invariant part of morphisms in $\Fuk \Etilde$.
The precise definition is given in
\cite[Section 18]{Seidel_PL}.

There are two important classes of
$\iota$-invariant Lagrangian submanifolds in $\Etilde$.
One of them, called of type (U),
is the inverse image
$$
 \Ltilde = \beta^{-1}(L) = \Ltilde_+ \coprod \Ltilde_-
$$
of a Lagrangian submanifold $L$
whose image by $\pi$ is contained in a simply-connected
domain inside $\bCx$ (i.e., $\bC$ minus the base point).
It is the disjoint union of two connected components
$\Ltilde_+$ and $\Ltilde_-$.
The other, called of type (B), is
the inverse image
$$
 \Deltatilde_\gamma = \beta^{-1}(\Delta_\gamma)
$$
of the Lefschetz thimble $\Delta_\gamma$
for a vanishing path $\gamma$.
It is a Lagrangian $n$-sphere in $\Etilde$.

For type (U) Lagrangian submanifolds
$\Ltilde_0$ and $\Ltilde_1$ of $\Etilde$,
their intersections are two disjoint copies of
intersections between $L_0$ and $L_1$ in $E$.
By taking $\iota$-invariant, one can show that
there is a natural isomorphism
$$
 \hom_{\frakF(\pi)}( \Ltilde_0, \Ltilde_1)
  \cong \hom_{\Fuk E}( L_0, L_1)
$$
of vector spaces,
which lifts to a cohomologically full and faithful $A_\infty$-functor
$$
 \Fuk E \to \frakF(\pi).
$$

For type (B) Lagrangian submanifolds,
the situation is a little more complicated,
but the conclusion is that
the full $A_\infty$-subcategory of $\frakF(\pi)$
consisting of
$
 \bsDeltatilde = (\Deltatilde_{\gamma_1}, \dots, \Deltatilde_{\gamma_m})
$
for a distinguished set $(\gamma_i)_{i=1}^m$ of vanishing paths
is quasi-isomorphic to the {\em directed subcategory}
$\dirFuk(\bsV)$ of $\Fuk E_0$,
whose set of objects is the distinguished basis
$
 \bsV = (V_{\gamma_1}, \dots, V_{\gamma_m})
$
of vanishing cycles,
whose spaces of morphisms are given by
$$
 \hom_{\dirFuk(\bsV)}(V_{\gamma_i}, V_{\gamma_j}) =
 \begin{cases}
   \bC \cdot \id_{V_{\gamma_i}} & i = j, \\
   \hom_{\Fuk E_0}(V_{\gamma_i}, V_{\gamma_j}) & i < j, \\
   0 & \text{otherwise},
 \end{cases}
$$
and non-trivial $A_\infty$-operations coincide
with those in $\Fuk E_0$.
We write this $A_\infty$-category as $\Fuk \pi$.
Although $\Fuk \pi$ depends on the choice
of a distinguished set of vanishing paths,
the derived category $D^b \Fuk \pi$ is independent of this choice
and gives an invariant of the Lefschetz fibration $\pi$.

\begin{figure}
\centering
\input{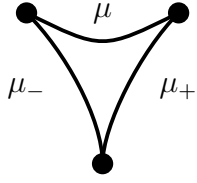}
\caption{A matching path}
\label{fg:matching_path}
\end{figure}


Let $\mu : [-1, 1] \to \bC$ be an embedded path in $\bC$ such that
$\mu^{-1}(\Critv (\pi)) = \{ -1, 1 \}$.
One can deform $\mu$ and split it into two pieces
$
 \mu_{\pm}(t) = \mu(\pm t)
$
to obtain a pair of vanishing paths
as shown in Figure \ref{fg:matching_path}.
If the vanishing cycles $V_{\mu_-}$ and $V_{\mu_+}$
are isotopic as exact framed Lagrangian $(n-1)$-spheres in $E_0$,
then $\mu$ is called a {\em matching path}.
In this case,
one can perturb $\Delta_{\mu_{+}} \cup \Delta_{\mu_{-}}$
to obtain a Lagrangian $n$-sphere $\Sigma_\mu$ in $E$
called the {\em matching cycle}.

Symplectic Picard-Lefschetz theory describes
the action of the symplectic Dehn-twist along a Lagrangian sphere
on the derived Fukaya category.
It follows that the type (U) Lagrangian submanifold
$
 \Sigmatilde_\mu = \beta^{-1}(\Sigma_\mu)
$
of $\Etilde$ coming from a matching path $\mu$
is isomorphic to the mapping cone
over the (unique up to scalar) non-trivial morphism
from $\Deltatilde_{\mu_-}$ to $\Deltatilde_{\mu_+}$
in the derived Fukaya category $D^b \frakF(\pi)$
of the Lefschetz fibration;
$$
 \Sigmatilde_\mu \cong \Cone( \Deltatilde_{\mu_-} \to \Deltatilde_{\mu_+} ).
$$
This is important
since it allows one to reduce Floer-theoretic computation
for matching cycles in $\Fuk E$
to that for vanishing cycles in $\Fuk E_0$.
By iterating this process,
one ends up with the case of symplectic 2-manifolds,
where Lagrangian submanifolds are simple closed curves
and the problem of counting holomorphic disks is purely combinatorial.

A natural source of matching paths is
a {\em Lefschetz bifibration}.
It is a diagram
\vspace{7mm}
\begin{equation*}
\begin{psmatrix}[colsep=1.5]
 \scE & \bC^2 & \bC
\end{psmatrix}
\psset{shortput=nab,arrows=->,labelsep=3pt,nodesep=2pt}
\small
\ncline{1,1}{1,2}_{\varpi}
\ncline{1,2}{1,3}_{\psi}
\ncarc[arcangle=30]{1,1}{1,3}^{\Psi = \psi \circ \varpi}
\end{equation*}
with certain genericity conditions,
which implies that
for any critical point of $\Psi$,
there are local holomorphic coordinates of $\scE$ and $\bC^2$
such that
$$
 \varpi(x_1, \dots, x_{2 n})
  = (x_1^2 + x_2^2 + \cdots + x_{2 n}^2, x_1),
 \qquad
 \psi(y_1, y_2) = y_1.
$$
Then the map
$$
 \scE_w \xto{\varpi_w} \scS_w
$$
from $\scE_w = \Psi^{-1}(w)$ to $\scS_w = \psi^{-1}(w)$
for a general $w \in \bC$ is a Lefschetz fibration,
and by chasing the trajectory
of critical values of $\varpi_w$
as $w$ varies along a vanishing path $\gamma$,
one obtains a matching path $\mu$ in $\scS_0$
such that the matching cycle $\Sigma_\mu$ is
Hamiltonian isotopic to the vanishing cycle $V_\gamma$.

\section{Homological mirror symmetry for $\bP^2$}
 \label{sc:p2}

We recall homological mirror symmetry for $\bP^2$
proved by Seidel \cite{Seidel_VC2} in this section.
The mirror of $\bP^2$
is given by the Laurent polynomial
$$
 W(x, y) = x + y + \frac{1}{x y},
$$
which has critical points
$
 (x, y) = (1, 1), \, (\omega, \omega), \, (\omega^2, \omega^2)
$
with critical values
$
 3, \, 3 \omega, \, 3 \omega^2.
$
Here
$
 \omega=\exp(2 \pi \sqrt{-1}/3)
$
is a primitive cubic root of unity.
Let $(\gamma_i)_{i=1}^3$ be the distinguished set of vanishing paths
obtained as the straight line segments
from the origin to the critical values of $W$
as shown in Figure \ref{fg:Z3_path}.
The corresponding vanishing cycles are denoted by $(C_i)_{i=1}^3$.

Consider the Lefschetz bifibration
\vspace{7mm}
\begin{equation} \label{eq:lefschetz_bifibration_P2}
\begin{psmatrix}[colsep=1.5]
 (\bCx)^{2} & \bC \times \bCx & \bC
\end{psmatrix}
\psset{shortput=nab,arrows=->,labelsep=3pt,nodesep=3pt}
\small
\ncline{1,1}{1,2}_{\varpi}
\ncline{1,2}{1,3}_{\psi}
\ncarc[arcangle=30]{1,1}{1,3}^{W = \psi \circ \varpi}
\end{equation}
where
$$
 \varpi(x, y) = \lb x + y + \frac{1}{x y}, y \rb
$$
and
$$
 \psi(u, v) = u.
$$
The critical points of
$$
 \varpi_t : W^{-1}(t) \to \psi^{-1}(t) \cong \Spec \bC[y, y^{-1}]
$$
are given by
\begin{gather*}
 2 x + \frac{1}{x^2} = t,
\end{gather*}
with critical values
$$
 y = \frac{1}{x^2}.
$$
The critical values are given by
$
 y = (-2)^{2/3}
$
at $t = 0$,
which moves as shown in Figure \ref{fg:Z3_sheet_vc}
along the vanishing paths $(\gamma_i)_{i=1}^3$.
These trajectories $(\mu_i)_{i=1}^3$
are matching paths
corresponding to $(C_i)_{i=1}^3$.
The fiber $W^{-1}(0)$ can be compactified to an elliptic curve
by adding one point over $y=0$ and two points over $y=\infty$.
The vanishing cycles on $W^{-1}(0)$
are shown in Figure \ref{fg:Z3_surface_vc}.

\begin{figure}[thp]
\begin{minipage}{.5 \linewidth}
\centering
\input{Z3_path.pst}
\caption{A path on the $W$-plane}
\label{fg:Z3_path}
\end{minipage}
\begin{minipage}{.5 \linewidth}
\centering
\input{Z3_sheet_vc.pst}
\caption{Matching paths on the $y$-plane}
\label{fg:Z3_sheet_vc}
\end{minipage}
\end{figure}

\begin{figure}[htbp]
\centering
\input{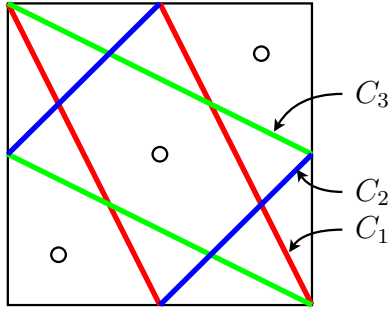}
\caption{Vanishing cycles on $W^{-1}(0)$}
\label{fg:Z3_surface_vc}
\end{figure}

On the mirror side,
one has the full exceptional collection
$$
 (E_1, E_2, E_3) = (\Omega_{\bP^2}^2(2)[2], \Omega_{\bP^2}(1)[1], \scO_{\bP^2})
$$
by Beilinson \cite{Beilinson}.
The endomorphism algebra is given by
$$
 \hom(E_i, E_j) =
\begin{cases}
 \wedge^{j-i} V & i \le j, \\
 0 & \text{otherwise},
\end{cases}
$$
where $V$ is a 3-dimensional vector space.
This endomorphism algebra is formal as an $A_\infty$-algebra
with respect to the standard enhancement of $D^b \coh \bP^2$.
One can easily see
that there is an isomorphism
$$
 \dirFuk W^{-1}(0) \to D^b \coh \bP^2
$$
of $A_\infty$-categories
sending $C_i$ to $E_i$ for $i=1,2,3$.
Indeed,
one can see in Figure \ref{fg:Z3_surface_vc}
that $C_i$ and $C_j$ for $i \ne j$ intersect at three points,
and six triangles bounded by $C_1$, $C_2$ and $C_3$
correspond to six non-zero compositions in
$$
 \hom(E_2, E_3) \otimes \hom(E_1, E_2) \to \hom(E_1, E_3).
$$
The torus-equivariant version of homological mirror symmetry for $\bP^2$
(and more generally for toric del Pezzo surfaces)
is discussed in detail in \cite{Ueda-Yamazaki_toricdP}.

\section{Homological mirror symmetry for $\bP^3$}
 \label{sc:p3}

The mirror of the projective space $\bP^3$
is given by the Laurent polynomial
$$
 W(x, y, z) = x + y + z + \frac{1}{x y z}
$$
with critical points
$
 x = y = z = \pm 1, \pm \sqrt{-1}
$
and critical values $\pm 4, \pm 4 \sqrt{-1}$.
Choose a distinguished set of vanishing paths
$( \gamma_i )_{i=1}^4$
as the straight line segments
from the origin to the critical values
as shown in Figure \ref{fg:P3_vanishing_paths},
and let $(C_i)_{i=1}^4$ be the corresponding
distinguished basis of vanishing cycles.
To use Picard-Lefschetz theory,
consider the Lefschetz bifibration
\vspace{7mm}
\begin{equation} \label{eq:lefschetz_bifibration1}
\begin{psmatrix}[colsep=1.5]
 (\bCx)^{3} & \bC \times \bCx & \bC
\end{psmatrix}
\psset{shortput=nab,arrows=->,labelsep=3pt,nodesep=3pt}
\small
\ncline{1,1}{1,2}_{\varpi}
\ncline{1,2}{1,3}_{\psi}
\ncarc[arcangle=30]{1,1}{1,3}^{W = \psi \circ \varpi}
\end{equation}
where
$$
 \varpi(x, y, z) = \lb x + y + z + \frac{1}{x y z}, z \rb
$$
and
$$
 \psi(u, v) = u.
$$
The critical points of
$$
 \varpi_t : W^{-1}(t) \to \psi^{-1}(t) \cong \Spec \bC[z, z^{-1}]
$$
are given by
\begin{gather*}
 x = y, \qquad 3 x + \frac{1}{x^3} = t,
\end{gather*}
with critical values
$$
z = \frac{1}{x^3}.
$$
The critical values are given by
$
 z = (-3)^{3/4}
$
at $t = 0$,
which moves as shown in Figure \ref{fg:P3_z-plane_mp}
along the vanishing paths $(c_i)_{i=1}^4$.
These trajectories $(\mu_i)_{i=1}^4$
are matching paths
corresponding to $(C_i)_{i=1}^4$.
Take $z = 1$ as a base point and
choose a distinguished set $(\delta_i)_{i=1}^4$
of vanishing paths for $\varpi_0$
as straight line segments from the base point
as shown in Figure \ref{fg:P3_z-plane_vp}.
The fiber $\varpi_0^{-1}(z)$ is a branched double cover of $\bCx$
by the $y$-projection
$$
\begin{array}{cccc}
 \pi_z : & \varpi_0^{-1}(z) & \to & \bCx \\
 & \vin & & \vin \\
 & (x, y, z) & \mapsto & y.
\end{array}
$$
Figure \ref{fg:P3_y-plane_mp} shows
the behavior of these branch points
along vanishing paths $(\delta_i)_{i=1}^4$,
which can be considered as matching paths
coming from the Lefschetz bifibration
\vspace{7mm}
\begin{equation} \label{eq:lefschetz_bifibration2}
\begin{psmatrix}[colsep=1.5]
 W^{-1}(0) & \bCx \times \bCx & \bCx \\
\end{psmatrix}
\psset{shortput=nab,arrows=->,labelsep=3pt}
\small
\ncline{1,1}{1,2}_{\pi}
\ncline{1,2}{1,3}_{\varphi}
\ncarc[arcangle=30]{1,1}{1,3}^{\varpi_0 = \varphi \circ \pi}
\end{equation}
where
$
 \pi(x, y, z) = (z, y)
$
and
$
 \varphi(z, y) = z.
$
Note that one has an inductive structure here,
as \eqref{eq:lefschetz_bifibration2} is almost identical to
\eqref{eq:lefschetz_bifibration_P2}.

One can see that
the number of intersection points of $C_i$ and $C_j$
for $i < j$ is equal to the dimension of $\wedge^{j-i} V$,
where $V$ is a vector space of dimension four.
As an example,
consider the intersection of $C_1$ and $C_2$.
The matching paths $\mu_1$ and $\mu_2$ intersect
at one critical value of $\varpi_0$ and
one regular value of $\varpi_0$.
The intersection of $C_1$ and $C_2$ over
the critical value of $\varpi_0$ consists of one point,
i.e., the critical point of $\varpi_0$.
The intersection of $C_1$ and $C_2$ over the regular value of $\varpi_0$
consist of three points, as one can see from Figure
\ref{fg:P3_y-plane_mp}
(cf. also Figures \ref{fg:Z3_sheet_vc}
and \ref{fg:Z3_surface_vc}).
As for the intersection of $C_1$ and $C_3$,
the corresponding matching paths intersect
at two regular points of $\varpi_0$,
and the intersection over each of them consists of three points.

\begin{figure}[htbp]
\begin{minipage}{.5 \linewidth}
\centering
\input{P3_vanishing_paths.pst}
\caption{A distinguished set of vanishing paths}
\label{fg:P3_vanishing_paths}
\end{minipage}
\begin{minipage}{.5 \linewidth}
\centering
\input{P3_z-plane_mp.pst}
\caption{Matching paths on the $z$-plane}
\label{fg:P3_z-plane_mp}
\end{minipage}
\begin{minipage}{.5 \linewidth}
\centering
\input{P3_z-plane_vp.pst}
\caption{A distinguished set of vanishing paths}
\label{fg:P3_z-plane_vp}
\end{minipage}
\begin{minipage}{.5 \linewidth}
\centering
\input{P3_y-plane_mp.pst}
\caption{Matching paths on the $y$-plane}
\label{fg:P3_y-plane_mp}
\end{minipage}
\begin{minipage}{.5 \linewidth}
\centering
\input{P3_z-plane_loop.pst}
\caption{A loop in the $z$-plane}
\label{fg:P3_z-plane_loop}
\end{minipage}
\begin{minipage}{.5 \linewidth}
\centering
\input{P3_y-palne_rotate.pst}
\caption{The behavior of branch points of $\pi_z$}
\label{fg:P3_y-plane_rotate}
\end{minipage}
\end{figure}

To use Picard-Lefschetz theory
to do computations in the Fukaya category of $W$,
consider the pull-back
$$
\begin{CD}
 W^{-1}(0)^{\sim} @>{\varpitilde_0}>> \bC \\
 @VVV @VV{\exp}V \\
 W^{-1}(0) @>{\varpi_0}>> \bCx
\end{CD}
$$
of $\varpi_0$ by the universal cover of the algebraic torus.
The existence of infinitely many critical points
for a given critical value does not cause any problem,
since the corresponding vanishing cycles do not intersect.
The passage from $W^{-1}(0)$ to $W^{-1}(0)^\sim$
can be taken into account
by noting that
as one goes counterclockwise around the origin in the $z$-plane
as shown in Figure \ref{fg:P3_z-plane_loop},
the branch points of $\pi_z$ rotates clockwise by $2 \pi / 3$
as in Figure \ref{fg:P3_y-plane_rotate}.

The universal cover of the $z$-plane is obtained by
cutting the $z$-plane along the dashed line
in Figure \ref{fg:P3_z-plane_vp}
and gluing infinitely-many copies of it.
We set the point $z = 1$ on the zero-th sheet
as the base point $*$ and
take a distinguished set of vanishing paths for $\varpitilde_0$
as in Figure \ref{fg:P3_z-plane_cover_vp}.

\begin{figure}
\begin{minipage}{\linewidth}
\centering
\input{P3_z-plane_cover_vp.pst}
\caption{Vanishing paths on the universal cover}
\label{fg:P3_z-plane_cover_vp}
\end{minipage}
\begin{minipage}{\linewidth}
\centering
\input{P3_z-plane_cover_mp.pst}
\caption{Matching paths on the universal cover}
\label{fg:P3_z-plane_cover_mp}
\end{minipage}
\end{figure}

Let $\Deltabar_1$, $\Deltabar_2$, and $\Deltabar_3$
be the vanishing cycles of $\varpi_0$
along the vanishing paths
$\delta_1$, $\delta_2$ and $\delta_3$ respectively.
We write the vanishing cycles of $\varpitilde_0$
along the vanishing paths $\delta_i$
in Figure \ref{fg:P3_z-plane_cover_vp}
as $\Delta_i$ for $i \in \bZ$.
Let further $\scB$ be the Fukaya category of $\varpi_0^{-1}(1)$
consisting of $\{ \Deltabar_i \}_{i=1}^3$ and
$\scBtilde$ be the Fukaya category of $\varpitilde_0^{-1}(*)$
consisting of $\{ \Delta_i \}_{i \in \bZ}$.
Then one has a quasi-equivalence
$$
 \scBtilde \simto \scB
$$
of $A_\infty$-categories sending $\Delta_i$
to $\Deltabar_{\ibar}$,
where $\ibar$ is $i$ modulo 3.
We write the directed subcategory of $\scBtilde$
with respect to the order
$$
 \Delta_i < \Delta_{j}, \qquad i < j
$$
as $\scAtilde$.
The spaces of morphisms between $\Delta_i$ can be written as
$$
 \hom_{\scAtilde}(\Delta_i, \Delta_j) =
  \begin{cases}
   \bC \cdot \id_i
    & i = j, \\
   \bC \cdot \id_{i, j} \oplus \bC \cdot \id^\vee_{i, j}
    & i < j \text{ and } j \equiv i \mod 3, \\
   \Vbar & i < j \text{ and } j \equiv i + 1 \mod 3, \\
   \wedge^2 \Vbar & i < j \text{ and } j \equiv i + 2 \mod 3, \\
   0 & i > j,
  \end{cases}
$$
where $\id_i$ is the unit and
$$
 \Vbar = \vspan \lc e_1, e_2, e_3 \rc
$$
is a vector space of dimension three.
One can show,
by direct counting of triangles just as in Section \ref{sc:p2},
that the $A_\infty$-operation $\frakm_2$ on the spaces of morphisms
is given by the wedge product,
where $\id_{i, j}$ and $\id^\vee_{i, j}$ are identified
with the elements $1 \in \wedge^0 \Vbar$ and
$e_1 \wedge e_2 \wedge e_3 \in \wedge^3 \Vbar$ respectively.
Higher $A_\infty$-operations on $\scAtilde$ are irrelevant
for the argument below.

Let $C_i$ for $i \in \bZ$ be the lift to $W^{-1}(0)^\sim$
of a vanishing cycle on $W^{-1}(0)$,
which corresponds to the matching path $\mu_i$
obtained by concatenating $\delta_i$ and $\delta_{i+3}$
as in Figure \ref{fg:P3_z-plane_cover_mp}.
Let further $\Fuk W^{-1}(0)^\sim$ be the Fukaya category of $W^{-1}(0)^\sim$
consisting of $\{ C_i \}_{i \in \bZ}$ and
$\dirFuk W^{-1}(0)^\sim$ be its directed subcategory
with respect to the order
$C_i < C_j$ for $i < j$.
By symplectic Picard-Lefschetz theory
recalled in Section \ref{sc:PL},
there is a cohomologically full and faithful functor
$$
 \Fuk W^{-1}(0)^\sim \to D^b \scAtilde,
$$
which maps the objects as
$$
 C_i \mapsto \Cone \lb \Delta_{i} \xto{\id_{i, i+3}} \Delta_{i+3} \rb.
$$

On the mirror side,
the passage to the universal cover of the $z$-plane
corresponds to working equivariantly
with respect to the subgroup
$$
 \bT_3 = \{ (\alpha, \beta, \gamma) \in \bT \mid \alpha = \beta = 1\}
$$
of the torus $\bT \cong (\bCx)^3$ acting on $\bP^3$ by
$$
\begin{array}{cccc}
 \bT \ni (\alpha, \beta, \gamma) : &
  \bP^3 & \to & \bP^3 \\
 & \vin & & \vin \\
 & [x_0:x_1:x_2:x_3] & \mapsto &
  [x_0:\alpha x_1:\beta x_2:\gamma x_3].
\end{array}
$$
The full exceptional collection
$$
 (E_1, E_2, E_3, E_4)
  = (\Omega_{\bP^3}^3(3)[3],
      \Omega_{\bP^3}^2(2)[2],
      \Omega_{\bP^3}^1(1)[1],
      \scO_{\bP^3})
$$
admits a natural $\bT$-linearization,
so that the endomorphism algebra is given by
$$
 \hom(E_i, E_j)
  = \begin{cases}
 \wedge^{j - i} V & i \le j, \\
 0 & \text{otherwise},
  \end{cases}
$$
with the natural $\bT$-action.
Moreover, this endomorphism algebra is formal
as an $A_\infty$-algebra with respect to a standard enhancement
of $D^b \coh^\bT \bP^3$.
Now it is easy to see that
there is an $A_\infty$-functor
$$
 \dirFuk W^{-1}(0)^\sim \to D^b \coh^{\bT_3} \bP^3
$$
sending $C_{i + 4 j}$ to $E_i \otimes \rho_j$,
where $\rho_j : \bT_3 \to \bCx$ for $j \in \bZ$ is
the one-dimensional representation
sending $(1, 1, \gamma) \in \bT_3$ to $\gamma^j$;
for example,
one has
\begin{align*}
 \hom(E_1 \otimes \rho_i, E_2 \otimes \rho_j)
  &= \begin{cases}
      \bC \cdot e_4 & j = i - 1, \\
      \Vbar & j = i, \\
      0 & \text{otherwise},
  \end{cases} \\
 \hom(E_1 \otimes \rho_i, E_3 \otimes \rho_j)
  &= \begin{cases}
      \Vbar \wedge e_4 & j = i - 1, \\
      \wedge^2 \Vbar & j = i, \\
      0 & \text{otherwise},
  \end{cases} \\
 \hom(E_1 \otimes \rho_i, E_4 \otimes \rho_j)
  &= \begin{cases}
      (\wedge^2 \Vbar) \wedge e_4 & j = i - 1, \\
      \wedge^3 \Vbar & j = i, \\
      0 & \text{otherwise},
  \end{cases}
\end{align*}
which exactly matches the computation
in the Fukaya category,
as we show for general $n$ in Section \ref{sc:induction}.
This suffices to show the equivalence
$$
 D^b \dirFuk W^{-1}(0)^\sim \cong D^b \coh^{\bT_3} \bP^3,
$$
which induces the equivalence
$$
 D^b \Fuk W \cong D^b \coh \bP^3
$$
by passing to the non-equivariant situation.

\section{Inductive description of the Fukaya category}
 \label{sc:induction}

The mirror of the projective space $\bP^n$
is given by the Laurent polynomial
\begin{equation} \label{eq:Pn_mirror}
 W(x_1, \ldots, x_n) = x_1 + \cdots + x_n + \frac{1}{x_1 \cdots x_n},
\end{equation}
with critical points
$$
 x_1 = \cdots = x_n = \zeta^{1 - i},
  \qquad \zeta = \exp(2 \pi \sqrt{-1} / (n + 1) ),
  \qquad i = 1, \dots, n + 1
$$
and critical values $(n + 1) \zeta^{1 - i}$.
Choose a distinguished set of vanishing paths
$( \gamma_i )_{i=1}^{n+1}$
as the straight line segments
from the origin to the critical values,
so that $\gamma_i(1) = \zeta^{1-i}$.
The Fukaya category of $W$
consisting of vanishing cycles
$C_i$ along $\gamma_i$ for $i = 1, \dots, n + 1$
will be denoted by $\Fuk W$.

\begin{theorem} \label{th:Pn_fuk}
The spaces of morphisms in $\Fuk W$ are given by
$$
 \hom(C_i, C_j) =
\begin{cases}
 \bC \cdot \id_{C_i} & i = j, \\
 \wedge^{j-i} V & i < j, \\
 0 & \text{otherwise},
\end{cases}
$$
where $V$ is an $(n+1)$-dimensional vector space and
an element of $\wedge^i V$ has degree $i$.
The $A_\infty$-operations $\frakm_k$ are given
by the wedge product for $k = 2$, and vanish for $k \ne 2$.
\end{theorem}

\begin{proof}
Consider the Lefschetz bifibration
\vspace{7mm}
\begin{equation} \label{eq:lefschetz_bifibration3}
\begin{psmatrix}[colsep=1.5]
 (\bCx)^{n} & \bC \times \bCx & \bC
\end{psmatrix}
\psset{shortput=nab,arrows=->,labelsep=3pt}
\small
\ncline{1,1}{1,2}_{\varpi}
\ncline{1,2}{1,3}_{\psi}
\ncarc[arcangle=30]{1,1}{1,3}^{W = \psi \circ \varpi}
\end{equation}
where
$$
 \varpi(x_1, \ldots, x_n)
  = \lb x_1 + \cdots + x_n + \frac{1}{x_1 \cdots x_n}, x_n \rb
$$
and
$$
 \psi(u, v) = u.
$$
The critical points of
$$
 \varpi_t : W^{-1}(t) \to \psi^{-1}(t)
  \cong \Spec \bC[x_n, x_n^{-1}]
$$
are given by
\begin{gather*}
 x_1 = \cdots = x_{n-1}, \qquad n x_1^{n+1} - t x_1^n + 1 = 0
\end{gather*}
with critical values
$$
 x_n = \frac{1}{x_1^n}.
$$
As one varies $t$ along the vanishing path $\gamma_1$
from $t = 0$ to $t = n + 1$,
two points
$
 x = \exp(\pm \pi / (n + 1) \sqrt{-1}) / \sqrt[n+1]{n}
$
from the set of solutions of
\begin{equation} \label{eq:xdisc}
 n x^{n+1} - t x^n + 1 = 0
\end{equation}
at $t = 0$ collide at $x = 1$ and $t = n + 1$,
while the absolute values of other points
remains to be smaller than these two points,
so that their behavior is as shown
in Figure \ref{fg:Pn_x-plane_mp}.
Here and below,
all figures are for $n = 4$,
but the general case is completely parallel.
The corresponding trajectory
of the critical values of $\varpi_t$ is shown
in Figure \ref{fg:Pn_z-plane_mp}.

Now consider the Lefschetz bifibration
\vspace{7mm}
\begin{equation} \label{eq:lefschetz_bifibration4}
\begin{psmatrix}[colsep=1.5]
 W^{-1}(0) & \bCx \times \bCx & \bCx \\
\end{psmatrix}
\psset{shortput=nab,arrows=->,labelsep=3pt}
\small
\ncline{1,1}{1,2}_{\pi}
\ncline{1,2}{1,3}_{\varphi}
\ncarc[arcangle=30]{1,1}{1,3}^{\varpi_0 = \varphi \circ \pi}
\end{equation}
where
$
 \pi(x_1, \ldots, x_n) = (x_{n}, x_{n-1})
$
and
$
 \phi(x_{n}, x_{n-1}) = x_n.
$
Take $x_n = 1$ as a base point and
choose a distinguished set $(\delta_i)_{i=1}^{n+1}$
of vanishing paths for $\varpi_0$
as the straight line segments from the base point
as shown in Figure \ref{fg:Pn_z-plane_vp}.
Consider the pull-back
$$
\begin{CD}
 W^{-1}(0)^{\sim} @>{\varpitilde_0}>> \bC \\
 @VVV @VV{\exp}V \\
 W^{-1}(0) @>{\varpi_0}>> \bCx
\end{CD}
$$
of $\varpi_0$ by the universal cover of the $x_n$-plane.
The $j$-th lift of the vanishing cycle $C_i \subset W^{-1}(0)$
to $W^{-1}(0)^{\sim}$ will be denoted by $C_{i+(n+1)j}$
for $i = 1, \dots, n+1$ and $j \in \bZ$.
We write the Fukaya category of $W^{-1}(0)^{\sim}$
consisting of $\lc C_i \rc_{i \in \bZ}$ as $\Fuk W^{-1}(0)^{\sim}$.

The universal cover of the $x_n$-plane is obtained
by gluing infinitely many copy
of the $x_n$-plane cut along the negative real axis.
We take the point $x_n = 1$ on the zeroth sheet
as the base point $*$ and take a distinguished set
$
 (\delta)_{i \in \bZ}
$
of vanishing paths as in Figure \ref{fg:Pn_z-plane_cover_vp}.
The vanishing cycle along $\delta_i$ will be denoted
by $\Delta_i$.
We write the directed Fukaya category of $\varpitilde_0$
consisting of $( \Delta_i )_{i \in \bZ}$ as $\scAtilde$.
The matching path corresponding to $C_i$ for $i \in \bZ$ is
obtained by concatenating $\delta_i$ and $\delta_{i+n}$
as in Figure \ref{fg:Pn_z-plane_cover_mp}.

\begin{figure}[htbp]
\begin{minipage}{.5 \linewidth}
\centering
\input{Pn_x-plane_mp.pst}
\caption{The behavior of solutions of \eqref{eq:xdisc}}
\label{fg:Pn_x-plane_mp}
\end{minipage}
\begin{minipage}{.5 \linewidth}
\centering
\input{Pn_z-plane_mp.pst}
\caption{Matching paths on the $x_n$-plane}
\label{fg:Pn_z-plane_mp}
\end{minipage}
\begin{minipage}{.5 \linewidth}
\centering
\input{Pn_z-plane_vp.pst}
\caption{Vanishing paths for $\varpi_0$}
\label{fg:Pn_z-plane_vp}
\end{minipage}
\begin{minipage}{.5 \linewidth}
\centering
\input{Pn_yy-plane_vp.pst}
\caption{Vanishing paths for $\Wbar$}
\label{fg:Pn_yy-plane_vp}
\end{minipage}
\begin{minipage}{\linewidth}
\centering
\input{Pn_z-plane_cover_vp.pst}
\caption{Vanishing paths on the universal cover}
\label{fg:Pn_z-plane_cover_vp}
\end{minipage}
\begin{minipage}{\linewidth}
\centering
\input{Pn_z-plane_cover_mp.pst}
\caption{Matching paths on the universal cover}
\label{fg:Pn_z-plane_cover_mp}
\end{minipage}
\end{figure}

Note that the fiber of $\varpi_0$ is isomorphic
to the fiber of
$$
\begin{array}{cccc}
 \Wbar : & (\bCx)^{n-1} & \to & \bC \\
 & \vin & & \vin \\
 & (x_1, \dots, x_{n-1}) & \mapsto &
  x_1 + \cdots + x_{n-1}
  + \displaystyle{\frac{1}{x_1 \cdots x_{n-1}}}
\end{array}
$$
by
$$
\begin{array}{ccc}
 \varpi_0^{-1}(x_n) & \to & \Wbar^{-1} \lb - x_n^{(n+1)/n} \rb \\
 \vin & & \vin \\
 (x_1, \dots, x_n) & \mapsto &
  x_n^{1/n} \lb x_1, \dots, x_{n-1} \rb.
\end{array}
$$
As $x_n$ varies along the vanishing paths
in Figure \ref{fg:Pn_z-plane_vp},
its image by the map $x \mapsto - x^{(n+1)/n}$ behaves
as in Figure \ref{fg:Pn_yy-plane_vp},
which are homotopic to the vanishing paths for $\Wbar$.
The fiber of $\pi_1: \varpi_0^{-1}(1) \to \bCx$
at $x_n = 1$ can be identified with
the fiber of
$
 \Wbar 
$
at $t = -1$,
which in turn can be identified
with the fiber of $\Wbar$ at the origin
by symplectic parallel transport.
Under this identification,
the vanishing paths $\delta_i$
in Figure \ref{fg:Pn_yy-plane_vp}
can be identified with the vanishing paths
$\gammabar_\ibar$
for $\Wbar$,
where $\ibar$ is $i$ modulo $n$.
It follows that
the vanishing cycle $\Delta_i$ along $\delta_i$ corresponds
to the vanishing cycle $\Cbar_\ibar$ along $\gamma_\ibar$.

Assume that the assertion of Theorem \ref{th:Pn_fuk}
holds for $\Wbar$,
so that one has
$$
 \hom_{\Fuk \Wbar}(\Cbar_i, \Cbar_j) =
  \begin{cases}
   \bC \cdot \id_i & i = j \\
   \wedge^{j-i} \Vbar & i < j, \\
   0 & \text{otherwise},
  \end{cases}
$$
where
$$
 \Vbar = \vspan \{ e_1, \dots, e_n \}
$$
is an $n$-dimensional vector space,
an element of $\wedge^k \Vbar$ has degree $k$, and
the $A_\infty$-operation is given by the wedge product.
Then one has
$$
 \hom_{\scAtilde}(\Delta_i, \Delta_j) =
  \begin{cases}
   \bC \cdot \id_i & i = j \\
   \wedge^0 \Vbar \oplus \wedge^n \Vbar & i < j
    \text{ and } j \equiv i \mod n, \\
   \wedge^{\overline{\jmath-\imath}} \Vbar & i < j
    \text{ and } j \not \equiv i \mod n, \\
   0 & \text{otherwise}
  \end{cases}
$$
as a vector space,
where $0 \le \overline{\jmath - \imath} < n$ is a representative
of $[j - i] \in \bZ / n \bZ$.
The gradings of $\Delta_i$ are chosen
so that an element of $\wedge^k V$ has
degree $k$.
The $A_\infty$-operations $\frakm_0$ and $\frakm_1$
vanish, and $\frakm_2$ is given by the wedge product
as
$$
 \frakm_2(\sigma, \tau) = (-1)^{\deg \tau} \sigma \wedge \tau.
$$
We write the elements of $\hom(\Delta_i, \Delta_{i+n})$
corresponding to $1 \in \wedge^0 \Vbar$ and
$e_1 \wedge \cdots \wedge e_n$
as $\id_{i,i+n}$ and $\id_{i,i+n}^\vee$ respectively.

By symplectic Picard-Lefschetz theory
recalled in Section \ref{sc:PL},
there is a cohomologically full and faithful functor
$$
 \Fuk W^{-1}(0)^\sim \to D^b \scAtilde,
$$
which maps the objects as
$$
 C_i
  \mapsto \lc \Delta_i \xto{\id_{i, i+n}} \Delta_{i+n} \rc.
$$
Then one has
\begin{align*}
 \hom(C_i, C_j)
  &= \hom \lb \lc
      \begin{CD}
       \Delta_i @>{\id_{i,i+n}}>> \Delta_{i+n}
      \end{CD}
     \rc,
     \lc
      \begin{CD}
       \Delta_{j} @>{\id_{j,j+n}}>> \Delta_{j+n}
      \end{CD}
     \rc \rb \\
  &= \lc
     \begin{CD}
 \hom(\Delta_{i+n}, \Delta_{j})
  @>{(-1)^{\deg \bullet-1} \frakm_2(\bullet, \, \id_{i,i+n})}>>
 \hom(\Delta_{i}, \Delta_{j}) \\
  @VV{- \frakm_2(\id_{j,j+n}, \, \bullet)}V
  @V{\frakm_2(\id_{j,j+n}, \, \bullet)}VV \\
 \hom(\Delta_{i+n}, \Delta_{j+n})
  @>{(-1)^{\deg \bullet-1} \frakm_2(\bullet, \, \id_{i,i+n})}>>
 \hom(\Delta_{i}, \Delta_{j+n}) 
     \end{CD}
    \rc,
\end{align*}
where the last line denotes the total complex
of the double complex.
If $j < i - 3$, then
every term in the last line of the right hand side is trivial.
If $i - n \le j \le i - 1$,
then the right hand side is given by
\begin{align*}
 \lc
\begin{CD}
 0 @>>> 0 \\
 @VVV @VVV \\
 0 @>>> \wedge^{j-i+n} \Vbar
\end{CD}
 \rc
\end{align*}
which is spanned by
$$
\lb
\begin{CD}
 @. \{ \Delta_i @>>> \Delta_{i+n} \} \\
 @. @V{\tau}VV @. \\
 \{ \Delta_j @>>> \Delta_{j+n} \} @.
\end{CD}
\rb
 \in \hom^1(C_i, C_j)
$$
for $\tau \in \wedge^{j-i+n} \Vbar$.
If $i = j$,
then the complex on the right hand side is given by
\begin{align*}
 \lc
\begin{CD}
 0 @>>> \bC \cdot \id_{\Delta_i} \\
 @VVV @VVV \\
 \bC \cdot \id_{\Delta_{i+n}} @>>>
 \bC \cdot \id_{i,i+n} \oplus \bC \cdot \id_{i,i+n}^\vee
\end{CD}
 \rc
\end{align*}
whose cohomology group is spanned by
$$
\lb \quad
\begin{CD}
 \{ \Delta_i @>>> \Delta_{i+n} \} \\
 @V{\id_{\Delta_i}}VV @V{\id_{\Delta_{i+n}}}VV \\
 \{ \Delta_i @>>> \Delta_{i+n} \}
\end{CD}
\ \rb
 \in \hom^0(C_i, C_j)
$$
and
$$
\lb
\begin{CD}
 @. \{ \Delta_i @>>> \Delta_{i+n} \} \\
 @. @V{\id_{i+n, i}^\vee}VV @. \\
 \{ \Delta_i @>>> \Delta_{i+n} \} @.
\end{CD}
\rb
 \in \hom^1(C_i, C_j).
$$
If $i + 1 \le j \le i + n - 1$,
then the complex on the right hand side is given by
\begin{align*}
 \lc
\begin{CD}
 0 @>>> \wedge^{j-i} \Vbar \\
 @VVV @VVV \\
 \wedge^{j-i} \Vbar @>>>  \wedge^{j-i} \Vbar
\end{CD}
 \rc
\end{align*}
whose cohomology group is spanned by
$$
\begin{CD}
 \{ \Delta_i @>>> \Delta_{i+n} \} \\
 @V{\tau}VV @V{(-1)^{j-i} \tau}VV \\
 \{ \Delta_i @>>> \Delta_{i+n} \}
\end{CD}
$$
for $\tau \in \wedge^{j-i} \Vbar$.
If $j = i + n$, then the complex on the right hand side is given by
\begin{align*}
 \lc
\begin{CD}
 \bC \cdot \id_{\Delta_{i+n}}
  @>>>
 \bC \cdot \id_{i+n, i+2n} \oplus \bC \cdot \id^\vee_{i+n, i+2n} \\
  @VVV @VVV \\
 \bC \cdot \id_{i+n, i+n} \oplus \bC \cdot \id^\vee_{i+n, i+2n} 
  @>>>
 \bC \cdot \id_{i, i+2n} \oplus \bC \cdot \id^\vee_{i, i+2n} 
\end{CD}
 \rc,
\end{align*}
whose cohomology group is spanned by
$$
\begin{CD}
 \{ \Delta_i @>>> \Delta_{i+n} \} \\
 @V{\id^\vee_{i,i+n}}VV @VV{(-1)^{n-1} \id^\vee_{i+n,i+2n}}V \\
 \{ \Delta_{i+n} @>>> \Delta_{i+2n} \}.
\end{CD}
$$
If $j > i + n$, then the complex on the right hand side is acyclic.

If we write
$$
 C_{i, j} = C_{i + (n + 1) j}, \qquad
  i = 1, \dots, n + 1 \text{ and } j \in \bZ,
$$
then the above calculation can be summarized as
$$
 \hom(C_{i, j}, C_{i', j'})
  = (\wedge^{i'-i} V \otimes \rho_{j'-j})^{\bT_n},
 \qquad 1 \le i < i' \le n + 1,
$$
where
$
 \bT_n = \bCx
$
is an algebraic torus,
$$
\begin{array}{cccc}
 \rho_i : & \bT_n & \to & \bCx \\
 & \vin & & \vin \\
 & \alpha & \mapsto & \alpha^i
\end{array}
$$
is an irreducible representation of $\bT_n$,
$$
 V = \rho_0 \oplus \cdots \oplus \rho_0 \oplus \rho_1,
$$
is an $(n + 1)$-dimensional representation of $\bT_n$, and
$\bullet^{\bT_n}$ denotes the subspace of $\bT_n$-invariants.

By descending from $W^{-1}(0)^\sim$ to $W^{-1}(0)$
and taking the directed subcategory,
one obtains
$$
 \hom_{\Fuk W}(C_i, C_j) =
  \begin{cases}
   \bC \cdot \id_{C_i} & i = j, \\
   \wedge^{j-i} V & i > j, \\
   0 & \text{otherwise}.
  \end{cases}
$$
It is straightforward to see that
the $A_\infty$-operation $\frakm_2$ on $\Fuk W$ is
given by wedge product.
One can also show,
either by direct calculation or
for degree reasons,
that $A_\infty$-operations $\frakm_k$ for $k \ne 2$ on $\Fuk W$
vanishes,
and Theorem \ref{th:Pn_fuk} is proved.
\end{proof}

In the proof of Theorem \ref{th:Pn_fuk},
we have thrown away the extra information
obtained by lifting from $W^{-1}(0)$
to its $\bZ$-cover $W^{-1}(0)^\sim$
at each step of the induction.
One can also keep this information,
and the resulting category can be described as follows:

\begin{theorem} \label{th:Pn_equiv_fuk}
Let
$$
 \Wtilde = W \circ \exp : \bC^n \to \bC
$$
be the pull-back of the mirror $W$ of $\bP^n$
by the $\bZ^n$-covering given by the exponential map
$$
 \exp : \bC^n \to (\bCx)^n.
$$
Let $C_{i, \bsj}$ denote the $\bsj$-th lift of $C_i$
for $i = 1, \dots, n + 1$ and $\bsj \in \bZ^n$.
Then one has
$$
 \hom(C_{i, \bsj}, C_{i', \bsj'}) =
\begin{cases}
 \bC \cdot \id_{C_{i, \bsj}} & i = i' \text{ and } \bsj = \bsj', \\
 (\wedge^{i'-i} V \otimes \rho_{\bsj' - \bsj})^\bT & i < i', \\
 0 & \text{otherwise},
\end{cases}
$$
where $V$ is an $(n+1)$-dimensional vector space
with an action of an algebraic torus $\bT = (\bCx)^n$ given by
$$
\begin{array}{cccc}
 \bT \ni (\alpha_1, \dots, \alpha_n) : & \bC^{n+1} & \to & \bC^{n+1} \\
 & \vin & & \vin \\
 & (x_0, x_1, \dots, x_n) & \mapsto
 & (x_0, \alpha_1 x_1, \dots, \alpha_n x_n),
\end{array}
$$
and
$$
\begin{array}{cccc}
 \rho_{\bsj} : & \bT & \to & \bCx \\
 & \vin & & \vin \\
 & (\alpha_1, \dots, \alpha_n) & \mapsto
 & (\alpha_1^{j_1}, \dots, \alpha_n^{j_n})
\end{array}
$$
is a one-dimensional representation of $\bT$
for $\bsj = (j_1, \dots, j_n)$.
\end{theorem}

The proof of Theorem \ref{th:Pn_equiv_fuk} is
completely parallel to that of Theorem \ref{th:Pn_fuk}.

\section{Tropical coamoeba}
 \label{sc:coamoeba}

We introduce the notion of a tropical coamoeba and
prove Theorem \ref{th:coamoeba} in this section.
A tropical coamoeba is a generalization
of a pair of a dimer model and an internal perfect matching on it
to higher dimensions.
See \cite{Ueda-Yamazaki_NBTMQ,
Ueda-Yamazaki_BTP,
Ueda-Yamazaki_toricdP,
Futaki-Ueda_A-infinity}
and references therein
for dimer models and its application to homological mirror symmetry.

\begin{definition} \label{def:tropical_coamoeba}
A {\em tropical coamoeba}
$
 G = ( (P_i)_{i=1}^m, \deg, \sgn)
$
of a Laurent polynomial
$
 W : (\bCx)^n \to \bC
$
consists of
\begin{itemize}
 \item
a polyhedral decomposition
$$
 T = \bigcup_{i=1}^m P_i,
$$
of a real $n$-torus $T = \bR^n / \bZ^n$
into an ordered set $(P_i)_{i=1}^m$ of polytopes,
 \item
a map
$$
 \deg : F_1 \to \bZ
$$
from the set $F_1$ of facets
to $\bZ$
called the {\em degree}, and
 \item
a map
$$
 \sgn : F_2 \to \{ 1, - 1 \}
$$
from the set $F_2$ of codimension two faces
called the {\em sign},
\end{itemize}
satisfying the following:
\begin{itemize}
 \item
There is a CW complex $Y$ in $W^{-1}(0)$ and
a deformation retraction
\begin{gather*}
 F : W^{-1}(0) \times [0, 1] \to W^{-1}(0), \\
 F(\bullet, 0) = \id_{W^{-1}(0)}, \quad
 \Image F(\bullet, 1) = Y, \quad
 F(\bullet, 1)|_Y = \id_Y,
\end{gather*}
such that the restriction of
$
 F(\bullet, 1)
$
to the union
of a distinguished basis $(C_i)_{i=1}^m$ of vanishing cycles
is a surjection onto $Y$.
 \item
The argument map
$\Arg : (\bCx)^n \to T$
induces a homeomorphism $Y \simto \bigcup_{f \in F_1} f$
into the union of facets.
 \item
The boundary of the polytope $P_i$ is
the image of the vanishing cycle $C_i$;
$$
 \Arg(F(C_i, 1)) = \partial P_i, \qquad i = 1, \dots, m.
$$
 \item
There is a natural one-to-one correspondence
between the set of common facets of $P_i$ and $P_j$
and intersection points of $C_i$ and $C_j$,
and the degree function is given
by the Maslov index of the intersection
with respect to suitable gradings of $W^{-1}(0)$ and
$(C_i)_{i=1}^m$.
 \item
For each codimension two face $e \in F_2$,
one has an $A_\infty$-operation
\begin{equation} \label{eq:A_infty_sgn}
 \m_k(f_1, \dots, f_k) = \sgn(e) f_0
\end{equation}
in the Fukaya category $\Fuk W$,
where $(f_0, f_1, \dots, f_k)$
is the set of facets around $e$,
identified with intersections of vanishing cycles as above.
Moreover, any non-trivial $A_\infty$-operation
in $\Fuk W$ comes from a codimension two face of $P_i$
in this way.
 \item
Let $\Wtilde = W \circ \exp$ be the pull-back of $W$
by the universal covering map $\exp : \bC^n \to (\bCx)^n$.
Then the pull-back $\Gtilde$ of $G$ to the universal cover $\bR^n \to T$
gives a tessellation of $\bR^n$,
which encodes the information of
$\Fuk \Wtilde$
in just the same way as above,
so that polytopes, facets, and codimension two faces
correspond to vanishing cycles of $\Wtilde$,
their intersection points, and $A_\infty$-operations
respectively.
\end{itemize}
\end{definition}

It follows from the definition
that if $G$ is a tropical coamoeba of $W$,
then one can associate a directed $A_\infty$-categories
$\scA_G$ whose set of objects, a basis of the space of morphisms,
and non-trivial $A_\infty$-operations on this basis are given by
the set of polytopes, the set of facets, and
the set of codimension two faces respectively,
which satisfies
$$
 \Fuk W \cong \scA_G.
$$
Moreover,
the directed $A_\infty$-category $\scA_\Gtilde$
associated with the pull-back $\Gtilde$ of $G$
to the universal cover is equivalent
to the Fukaya category associated with $\Wtilde$;
$$
 \Fuk \Wtilde \cong \scA_\Gtilde.
$$

Now we prove Theorem \ref{th:coamoeba}.
We first discuss the case of $\bP^2$
along the lines of \cite{Ueda-Yamazaki_toricdP}.
The mirror of $\bP^2$ is given by
$$
 W(x, y) = x + y + \frac{1}{x y},
$$
which has three critical values
$3$, $3 \omega$ and $3 \omega^2$.
Choose a distinguished set $(c_i)_{i=1}^3$ of vanishing paths
as the straight line segments
from the origin to each critical values
as in Figure \ref{fg:Z3_vanishing_paths}.
The $y$-projection
$$
\begin{array}{cccc}
 \varpi_t : & W^{-1}(t) & \to & \bCx \\
 & \rotatebox{90}{$\in$} & & \rotatebox{90}{$\in$} \\
 & (x, y) & \mapsto & y
\end{array}
$$
has three branch points,
which moves as shown in Figure \ref{fg:Z3_sheet_vctwo}
along the vanishing paths.
The trajectories of these branch points
are images of vanishing cycles
by $\varpi = \varpi_0$.
There are six disks in $W^{-1}(0)$
bounded by these vanishing cycles,
which are projected onto three triangles
in Figure \ref{fg:Z3_sheet_vctwo}.
By contracting these six disks,
one obtains a graph on $W^{-1}(0)$
whose $\pi$ projection is shown
in Figure \ref{fg:Z3_sheet_contractionone}.
Figure \ref{fg:Z3_coamoeba_contraction}
shows a schematic picture
of the image of this graph
by the argument map.
Here, the color scheme in Figures \ref{fg:Z3_sheet_contractionone}
and \ref{fg:Z3_coamoeba_contraction}
is not a continuation of the scheme introduced
in Figures \ref{fg:Z3_vanishing_paths} and \ref{fg:Z3_sheet_vctwo}.
The horizontal and the vertical axes
in Figure \ref{fg:Z3_coamoeba_contraction}
correspond to $\arg y$ and $\arg x$ respectively.
The inverse image of the circle on the $y$-plane
in Figure \ref{fg:Z3_sheet_contractionone} by $\varpi_0$
is a non-trivial double cover of it,
which maps to a cycle in the class
$(2, -1) \in H_1(T, \bZ) \cong \bZ^2$
shown in black in Figure \ref{fg:Z3_coamoeba_contraction}.
Three legs in Figure \ref{fg:Z3_sheet_contractionone}
connect two branches of the double cover $\varpi_0$,
which map to vertical line segments
in Figure \ref{fg:Z3_coamoeba_contraction}.
As a result,
one obtains the division of $T$ into three hexagons
as shown in Figure \ref{fg:Z3_graph}.
It is easy to see that
the set of edges in Figure \ref{fg:Z3_graph}
corresponds to the set of intersection points of vanishing cycles,
and the set of nodes
corresponds to holomorphic disks bounded by vanishing cycles.
The colors of the nodes correspond
to the signs of the $A_\infty$-operations.

\begin{figure}[htbp]
\begin{tabular}[b]{cc}
\begin{minipage}{.45 \linewidth}
\centering
\input{Z3_vanishing_paths.pst}
\caption{A distinguished set of vanishing paths}
\label{fg:Z3_vanishing_paths}
\end{minipage}
&
\begin{minipage}{.45 \linewidth}
\centering
\input{Z3_sheet_vc2.pst}
\caption{The trajectories of the branch points}
\label{fg:Z3_sheet_vctwo}
\end{minipage}
\\
\ \vspace{0mm}\\
\begin{minipage}{.45 \linewidth}
\centering
\input{Z3_sheet_contraction1.pst}
\caption{Contracting $W^{-1}(0)$}
\label{fg:Z3_sheet_contractionone}
\end{minipage}
&
\begin{minipage}{.45 \linewidth}
\centering
\input{Z3_coamoeba_contraction.pst}
\caption{Image of the contraction by the argument map}
\label{fg:Z3_coamoeba_contraction}
\end{minipage}
\\
\ \vspace{0mm}\\
\multicolumn{2}{c}
{
\begin{minipage}{.45 \linewidth}
\centering
\input{Z3_graph.pst}
\caption{The honeycomb tiling}
\label{fg:Z3_graph}
\end{minipage}
}
\end{tabular}
\end{figure}

\begin{figure}
\begin{minipage}{.5 \linewidth}
\centering
\input{P3_sheet_contraction.pst}
\caption{Contraction on the $z$-plane}
\label{fg:P3_sheet_contraction}
\end{minipage}
\begin{minipage}{.5 \linewidth}
\centering
\input{P3_hexagon_motion.pst}
\caption{The monodromy around the origin}
\label{fg:P3_hexagon_motion}
\end{minipage}
\end{figure}
\begin{figure}[htbp]
\centering
\input{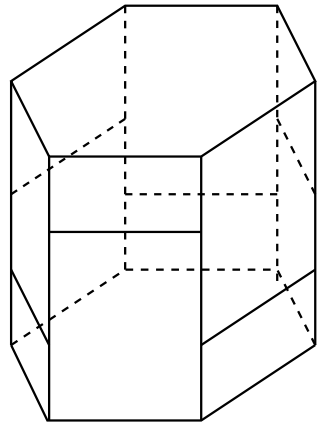}
\caption{A truncated octahedron}
\label{fg:P3_polytope}
\end{figure}
\begin{figure}[htbp]
\centering
\input{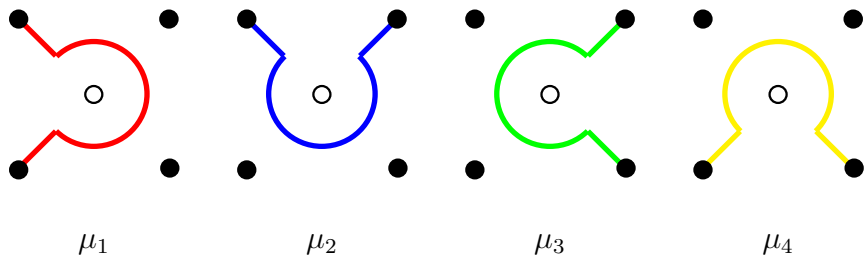}
\caption{Contractions of the matching paths}
\label{fg:P3_vc_contraction}
\end{figure}
\begin{figure}[htbp]
\centering
\input{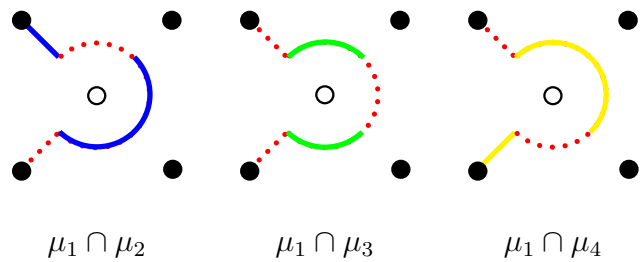}
\caption{Intersections of contracted matching paths}
\label{fg:P3_vc_contraction_int}
\end{figure}
\begin{figure}[htbp]
\begin{tabular}[b]{cc}
\begin{minipage}{.5 \linewidth}
\centering
\input{P3_polytope_int1.pst}
\caption{The facets of $P_1$ adjacent to $P_2$}
\label{fg:P3_polytope_int1}
\end{minipage}
&
\begin{minipage}{.5 \linewidth}
\centering
\input{P3_polytope_int2.pst}
\caption{The facets of $P_1$ adjacent to $P_3$}
\label{fg:P3_polytope_int2}
\end{minipage}
\\
\ \vspace{0mm}\\
\begin{minipage}{.5 \linewidth}
\centering
\input{P3_polytope_int3.pst}
\caption{The facets of $P_1$ adjacent to $P_4$}
\label{fg:P3_polytope_int3}
\end{minipage}
&
\begin{minipage}{.5 \linewidth}
\centering
\input{P3_polytope_int4.pst}
\caption{The edges of $P_1$ adjacent to $P_2$ and $P_3$}
\label{fg:P3_polytope_int4}
\end{minipage}
\end{tabular}
\end{figure}

Now we discuss the case of $\bP^3$.
By contracting the matching paths
in Figure \ref{fg:P3_z-plane_mp},
one obtains a circle with four legs
shown in Figure \ref{fg:P3_sheet_contraction}.
The fiber of $\varpi_0$ over a point on this circle is
symplectomorphic to $\Wbar^{-1}(0)$,
which can be contracted to the honeycomb graph
in Figure \ref{fg:Z3_graph} as explained above.
As one goes around the circle,
this honeycomb graph undergoes a monodromy
$$
 D_1 \mapsto D_2 \mapsto D_3 \mapsto D_1
$$
of order three,
where $D_i$ is the face in the honeycomb graph
corresponding to the $i$-th vanishing cycle of $\varpi_0$.
The image by the argument map
of this honeycomb graph bundle over the circle on the $z$-plane
divides the 3-torus $T$ into an obliquely-embedded hexagonal cylinder.
Four legs in Figure \ref{fg:P3_sheet_contraction} give
four faces perpendicular to the $\arg z$-axis,
which cut this hexagonal cylinder
into four truncated octahedra
$(P_i)_{i=1}^4$.

A {\em truncated octahedron} is the polytope
with fourteen faces, thirty-six edges and twenty-four vertices,
which is obtained from an octahedron
by truncating at its six vertices.
One of the four truncated octahedra in $T$ is shown
in Figure \ref{fg:P3_polytope},
where we have chosen to draw $\arg x$ and $\arg y$ horizontally,
and $\arg z$ vertically.
By pulling back this division of $T$ into four truncated octahedra
to the universal cover $\bR^3 \to T$,
one obtains the {\em bitruncated cubic honeycomb},
which is the Voronoi tessellation
for the body-centered cubic lattice.

It is straightforward to see that
intersections of vanishing cycles and
$A_\infty$-operations in Fukaya category correspond
to faces and edges of truncated octahedra respectively,
so that the decomposition of $T$ into four
truncated octahedra, together with a suitable choice
of the functions $\mu$ and $\sgn$,
gives a tropical coamoeba of $W$:
Matching paths are contracted
as in Figure \ref{fg:P3_vc_contraction},
and Figure \ref{fg:P3_vc_contraction_int} shows
the intersections of the matching path $\mu_1$ for $C_1$
with three other matching paths.
These intersections correspond
to faces of $P_1$ shown in Figures \ref{fg:P3_polytope_int1},
\ref{fg:P3_polytope_int2}, and \ref{fg:P3_polytope_int3},
which can be seen to be in natural bijection
with intersection points of $C_1$ with $C_2$, $C_3$ and $C_4$
by comparing with the discussion in Section \ref{sc:p3}.
It is also straightforward to see
that the edges of $P_i$ corresponds to
$A_\infty$-operations in $\Fuk \Wtilde$;
for example,
twelve edges corresponding to
$$
 \frakm_2 : \hom^1(C_2, C_3) \otimes \hom^1(C_1, C_2)
  \to \hom^2(C_1, C_3)
$$
are shown in Figure \ref{fg:P3_polytope_int4}.

Now we discuss the general case.
The {\em permutohedron} of order $n + 1$ is an $n$-dimensional polytope
lying on the hyperplane
\begin{equation*} 
 H = \lc (x_1, \dots, x_{n+1}) \in \bR^{n+1} \left|
  x_1 + \dots + x_{n+1} = \frac{n (n + 1)}{2} \right. \rc,
\end{equation*}
defined as the convex hull of the orbit of
$(1, 2, \dots, n+1) \in \bR^{n+1}$
under the action of the symmetric group $\frakS_{n+1}$
by permutations of coordinates.
Note that
the permutohedron of order three is a hexagon, and
the permutohedron of order four is a truncated octahedron.
A facet of a permutohedron of order $n$ corresponds
to a division
$$
 B_1 \sqcup B_2 = \{ 1, 2, \dots, n + 1 \}
$$
of the set $\{ 1, 2, \dots, n + 1 \}$
into the disjoint union of two subsets,
and a codimension two face corresponds to a division
$$
 B_1 \sqcup B_2 \sqcup B_3= \{ 1, 2, \dots, n + 1 \}
$$
into the disjoint union of three subsets.
The facet corresponding to the division $B_1 \sqcup B_2$
is given by
$$
 \sum_{i \in B_1} x_i = 1 + 2 + \cdots + \# B_1,
$$
and the codimension two face corresponding to the division
$B_1 \sqcup B_2 \sqcup B_3$ is given by
\begin{align*}
 \sum_{i \in B_1} x_i &= 1 + 2 + \cdots + \# B_1, \\
 \sum_{i \in B_1 \sqcup B_2} x_i
  &= 1 + 2 + \cdots + \# (B_1 \sqcup B_2), 
\end{align*}
so that the inclusion of a face into a facet corresponds to
a subdivision of a division of length two
into a division of length three.
The translations of the permutohedron of order $n + 1$
by the lattice of rank $n$
generated by
$$
 \ell_i = (n+1) e_i - (e_1 + \cdots + e_{n+1}),
  \qquad i = 1, \dots, n+1,
$$
where $e_i$ is the $i$-th coordinate vector,
tessellates the hyperplane H.
The polytope adjacent to the permutohedron
through the facet corresponding to the division $B_1 \sqcup B_2$ is
the translate of the permutohedron by
$$
 \sum_{i \in B_2} \ell_i.
$$
Every codimension two face of this tessellation is adjacent
to three facets, corresponding to
$B_1 \sqcup B_2$, $B_1' \sqcup B_2'$ and $B_1'' \sqcup B_2''$
such that $B_2'' = B_2 \sqcup B_2'$.

The set of facets of the permutohedron of order $n + 1$
maps bijectively to a basis of
$
 \wedge^\bullet V / (\wedge^0 V \oplus \wedge^{n+1} V)
$
by
$$
 B_1 \sqcup B_2 \mapsto \wedge_{i \in B_2} e_i.
$$
Under this correspondence,
the translates of three facets share a codimension two face
if and only if they correspond to
$u$, $v$ and $w$ in
$
 \wedge^\bullet V / (\wedge^0 V \oplus \wedge^{n+1} V)
$
such that $w = \pm u \wedge v$.

Now we inductively show
that the quotient of the above tessellation
by the lattice $\Lambda \cong \bZ^n$ generated by
$$
 \ell_i + (\ell_1 + \cdots + \ell_{n}), \qquad i = 1, \dots, n
$$
is a tropical coamoeba for the mirror of $\bP^{n}$.
By contracting the union of the matching paths
for $\varpi_0$ on the $x_{n}$-plane,
one obtains a circle $S$ with $n+1$ legs
$\{ l_1, \dots, l_{n+1} \}$,
numbered clockwise.
The fiber over a point on $S$ can be contracted
to the union of $n$ permutohedra
$\{ \Pbar_i \}_{i=1}^n$ of order $n$
by induction hypothesis,
which undergoes the cyclic monodromy
$$
 \Pbar_i \mapsto \Pbar_{i+1}, \qquad i = 1, \dots, n
$$
as one goes around the circle.
Its image by the argument map
gives a division of $T^n$ into an oblique cylinder
over $\Pbar_1$, which is divided into $n + 1$ permutohedra
$\{ P_i \}_{i=1}^{n+1}$ of order $n + 1$
by the $n + 1$ facets
coming from $n+1$ legs:
Let us call the direction of $\arg x_n$ vertical
and other directions horizontal.
The $x_n$-projection of the contracted vanishing cycle
consists of two legs $l_i$, $l_{i+n}$
and the part of the circumference between them.
The horizontal facets corresponding to $l_i$ and $l_{i+n}$
corresponds to $e_{n+1}$ and
$e_1 \wedge \dots \wedge e_n$ respectively.
There are $2^n - 2$ vertical facets of the cylinder,
and the one corresponding to
$$
 e_{i_1} \wedge \cdots \wedge e_{i_r}
$$
is divided into two,
one corresponding to
$$
 e_{i_1} \wedge \cdots \wedge e_{i_r}
$$
and the other corresponding to
$$
 e_{i_1} \wedge \cdots \wedge e_{i_r} \wedge e_{n+1}.
$$
As a whole, one obtains $2^{n+1} - 2$ facets,
and $P_1$ can be identified with
the permutohedron of order $n+1$.
Under this identification,
$P_i$ can be identified with the translation of $P_1$ by $\ell_{n+1}$,
and the union $\bigcup_{i=1}^{n+1} P_i$ is a fundamental region
of the lattice $\Lambda$.
The degree function takes the value $|B_2|$
on the facet corresponding to the division
$B_1 \sqcup B_2$,
and the value $\sgn(B_2', B_2)$ of the sign function
on the codimension two face $f$ of $P_i$,
where the facet of $P_i$
corresponding to $B_1 \sqcup B_2$ intersects $P_j$
and the facet of $P_j$ corresponding
corresponding to the division $B_1' \sqcup B_2'$ intersects $P_k$
for $i < j < k$,
is given by
$$
 \wedge_{i \in B_2 \sqcup B_2'} e_i
  = \sgn(B_2', B_2) \cdot (-1)^{|B_2|}
     (\wedge_{i \in B_2'} e_i) \wedge (\wedge_{i \in B_2} e_i).
$$

The $A_\infty$-category $\scA_G$
associated with the tropical coamoeba
$$
 G = ((P_i)_{i=1}^{n+1}, \deg, \sgn)
$$
defined above is quasi-equivalent
to the full subcategory of
a standard differential graded enhancement of $D^b \coh \bP^{n}$
consisting of
$$
 (E_1, E_2, \dots, E_{n+1})
  = (\Omega_{\bP^{n}}^{n}(n)[n],
      \Omega_{\bP^{n}}^{n-1}(n-1)[n-1], \dots,
      \scO_{\bP^{n}}).
$$
This implies the equivalence
$$
 D^b \scA_G \cong D^b \coh \bP^{n}
$$
of triangulated categories,
since $(E_1, \dots, E_{n+1})$ is a full exceptional collection
by Beilinson \cite{Beilinson}.
It is clear that
this equivalence lifts to the equivalence
$$
 D^b \scA_\Gtilde \cong D^b \coh^\bT \bP^{n}
$$
by sending the object of $\scA_\Gtilde$
corresponding to the $\bsj$-th lift of $P_i$
for $\bsj \in \Lambda \cong \bZ^{n}$
to $E_i \otimes \rho_{\bsj}$,
and Theorem \ref{th:coamoeba} is proved.
Theorem \ref{th:equiv_hms} is an immediate consequence of
Theorem \ref{th:coamoeba},
which in turn implies Corollary \ref{cor:hms}
just as in the two-dimensional case \cite{Ueda-Yamazaki_toricdP}.

\bibliographystyle{amsalpha}
\bibliography{bibs.bib}

\noindent
Masahiro Futaki

Department of Mathematics,
Graduate School of Science,
Kyoto University,
Kyoto 606-8502,
Japan

{\em e-mail address}\ : \  futaki@math.kyoto-u.ac.jp

\ \\

\noindent
Kazushi Ueda

Department of Mathematics,
Graduate School of Science,
Osaka University,
Machikaneyama 1-1,
Toyonaka,
Osaka,
560-0043,
Japan.

{\em e-mail address}\ : \  kazushi@math.sci.osaka-u.ac.jp
\ \vspace{0mm} \\

\end{document}